\documentclass[11pt]{amsart}
\usepackage{amsmath, amsthm, amsfonts, amssymb}
\usepackage{enumerate}
\usepackage[bookmarks=false]{hyperref}
\usepackage{xcolor}
\usepackage{verbatim}

\usepackage{comment}

\newtheorem{thm}{Theorem}[section]

\newtheorem{col}[thm]{Corollary}
\newtheorem{lem}[thm]{Lemma}

\newtheorem*{claim}{Claim}
\theoremstyle{definition}
\newtheorem{defn}[thm]{Definition}
\newtheorem{defn/lem}[thm]{Definition/Lemma}
\newtheorem{rmk}[thm]{Remark}

\newtheorem{examps}[thm]{Examples}

\newcommand{\C}{\ensuremath{\mathbb{C}}}

\newcommand{\M}{\ensuremath{\mathbb{M}}}
\newcommand{\N}{\ensuremath{\mathbb{N}}}
\newcommand{\Z}{\ensuremath{\mathbb{Z}}}

\newcommand{\T}{\ensuremath{\mathbb{T}}}
\newcommand{\cH}{\ensuremath{\mathcal{H}}}
\newcommand{\cU}{\ensuremath{\mathcal{U}}}
\newcommand{\cT}{\ensuremath{\mathcal{T}}}

\newcommand{\sk}{^{(k)}}
\newcommand{\si}{^{(\infty)}}
\newcommand{\sr}{^{(r)}}

\newcommand{\Span}{\mathrm{span}}
\newcommand{\HNN}{\mathrm{HNN}}
\newcommand{\Aut}{\mathrm{Aut}}

\begin{document}

\title[Toeplitz exactness for strong convergence]{Toeplitz exactness for strong convergence}

\author[D. Gao]{David Gao}
\address{Department of Mathematical Sciences, UCSD, 9500 Gilman Dr, La Jolla, CA 92092, USA}\email{weg002@ucsd.edu}\urladdr{https://sites.google.com/ucsd.edu/david-gao}

\author[S. Kunnawalkam Elayavalli]{Srivatsav Kunnawalkam Elayavalli}
\address{Department of Mathematics, UMD, Kirwan Hall, Campus Drive, MD 20770, USA}\email{sriva@umd.edu}
\urladdr{https://sites.google.com/view/srivatsavke/home}

\begin{abstract}
    We prove a new ``\emph{Toeplitz exactness}'' theorem for strong convergence. This is a machine to upgrade strong convergence in the general setting of $C^\ast$-correspondences, and has several applications.
\end{abstract}

\maketitle

\begin{center}
    \textit{for Marius and Leonel}
\end{center}

\section{Introduction}

Recent years have witnessed substantial advancements in the programs of strong convergence of unitary representations \cite{HaagerupThorbjornsen2005, vanhandel2025strongconvergenceshortsurvey, magee2025strong} and the free structure of $C^\ast$-algebras \cite{robertselfless, amrutam2025strictcomparisonreducedgroup, ozawa2025proximalityselflessnessgroupcalgebras}. The current paper is particularly inspired by the recent development \cite{gao2026newsourcepurelyfinite} that found a new connection. Our main result here reveals a general ``\emph{Toeplitz exactness}'' phenomenon that upgrades strong convergence to Toeplitz-Pimsner algebras associated to general $C^\ast$-correspondences, constructed in Pimsner's work \cite{pimsner1997class}. This is \emph{not} to be confused with preservation of ``on the nose exactness'' for Toeplitz-Pimsner algebras in the sense of \cite{dykema2001exactness}. Our main result is the following:

\begin{thm}\label{thm: main theorem}
    Fix index sets $I$ and $J$. Let $(X_i\sk)_{i \in I}$ be a tuple in a $C^\ast$-algebra $A\sk$ and $(\xi_j\sk)_{j \in J}$ be a tuple in a $C^\ast$-correspondence $\cH\sk$ over $A\sk$, for each $k \in \N \cup \{\infty\}$. Assume further that $A\si$ is generated by $(X_i\si)_{i \in I}$ and $\cH\si$ is generated by $(\xi_j\si)_{j \in J}$ as a $C^\ast$-correspondence over $A\si$. If $(X_i\sk; \xi_j\sk)_{i \in I, j \in J}$ converges strongly to $(X_i\si; \xi_j\si)_{i \in I, j \in J}$, then $(X_i\sk, T_{\xi_j}\sk)_{i \in I, j \in J}$ in $\cT(\cH\sk)$ converges strongly to $(X_i\si, T_{\xi_j}\si)_{i \in I, j \in J}$ in $\cT(\cH\si)$.
\end{thm}

Our approach uses the \emph{gauge invariant-uniqueness} theorem of Fowler--Muhly--Raeburn \cite{FowlerMuhlyRaeburn2003} for Toeplitz-Pimsner algebras associated to arbitrary $C^\ast$-correspondences, which completed a line of research beginning in the visionary work of Pimsner \cite{pimsner1997class}. We now demonstrate the utility of Theorem \ref{thm: main theorem} by quickly listing a few applications. Concerning the next result, recall the free exactness results of Skoufranis \cite{skoufranis2015notion} and Pisier \cite{PisierRDP} (see also Shlyakhtenko's appendix in \cite{male2011norm}). Recall that, for instance, Pisier's \cite[Theorem 7.1]{PisierRDP} says roughly that taking amalgamated free products over a fixed subalgebra preserves strong convergence. The following corollary recovers and generalizes all these results with a new proof:

\begin{col}\label{col: amalgamated free products}
    Fix an index set $J$ and an index set $I_j$ for each $j \in J$. Let $D\sk \subset A_j\sk$ be an inclusion of $C^\ast$-algebras with nondegenerate condition expectation $E_j\sk$; and let $(X_{i, j}\sk)_{i \in I_j}$ be a tuple in $A_j\sk$, for each $j \in J$ and $k \in \N \cup \{\infty\}$. Assume further that $A_j\si$ is generated by $(X_{i, j}\si)_{i \in I_j}$ for every $j \in J$. If $(X_{i, j}\sk; E_j\sk)_{i \in I_j} \to (X_{i, j}\si; E_j\si)_{i \in I_j}$ strongly for all $j \in J$, consistently on $D\sk$, then $(X_{i, j}\sk; \ast_{D\sk}^{j \in J} E_j\sk)_{i \in I_j, j \in J}$ in $\ast_{D\sk} A_j\sk$ converges strongly to $(X_{i, j}\si; \ast_{D\si}^{j \in J} E_j\si)_{i \in I_j, j \in J}$ in $\ast_{D\si} A_j\si$.
\end{col}

We note that Pisier's proof crucially used the analytic Khintchine type inequality of Ricard-Xu \cite{KhinRX} from noncommutative $L^p$--space theory. Our proof avoids this completely, and is based on a conceptually different \emph{algebraic} approach. As a side remark, note that our work provides also a satisfactory answer to question (3) of Section 5 in Skoufranis' work \cite{skoufranis2015notion}. In particular, our work is able to circumvent the technical challenges that Skoufranis faced in his approach \cite{skoufranis2015notion}, and extend the result significantly.

Another advantage of the general setup of Corollary \ref{col: amalgamated free products} is that it could be used to aid in proving the existence of strongly converging unitary representations, following the novel approach of \cite{gao2026newsourcepurelyfinite}. Particularly, in one of the steps in the proof of Lemma 2.3 in \cite{gao2026newsourcepurelyfinite}, one could alternatively use Corollary \ref{col: amalgamated free products} in place of an isomorphism theorem for Toeplitz-Pimnser algebras \cite[Theorem 2.2]{gao2026newsourcepurelyfinite} (in Remark \ref{rmk: application to groups} the argument is explained for the benefit of the reader). However, recall that the proof of the isomorphism theorem \cite[Theorem 2.2]{gao2026newsourcepurelyfinite} is much easier than the proof of our Corollary \ref{col: amalgamated free products}, rendering the original application of Toeplitz-Pimsner machinery in \cite[Lemma 2.3]{gao2026newsourcepurelyfinite} much more succinct and transparent. While our results in this paper cannot produce finite-dimensional models for strong convergence on their own, they can assist in the above manner owing to the apparatus developed in \cite{gao2026newsourcepurelyfinite}.

Below we document another application of upgrading strong convergence in the setting of HNN extensions (\cite{Ueda2008Remarks, gao2026selflessreducedamalgamatedfree}).

\begin{col}\label{col: HNN extensions}
    Fix an index set $I$. Let $A\sk$ be a $C^\ast$-algebra equipped with a faithful state $\rho\sk$, $B_1\sk, B_{-1}\sk \subset A\sk$ be subalgebras admitting state-preserving conditional expectations $E_1\sk$ and $E_{-1}\sk$, resp., and $\theta\sk: B_1\sk \to B_{-1}\sk$ be a state-preserving $\ast$-isomorphism; and let $(X_i\sk)_{i \in I}$ be a tuple in $A\sk$, for each $k \in \N \cup \{\infty\}$. Assume further that $A\si$ is generated by $(X_i\si)_{i \in I}$. If $(X_i\sk; \rho\sk, E_1\sk, E_{-1}\sk, \theta\sk)_{i \in I} \to (X_i\si; \rho\si, E_1\si, E_{-1}\si, \theta\si)_{i \in I}$ strongly, then $(X_i\sk, w\sk)_{i \in I}$ converges strongly to $(X_i\si, w\si)_{i \in I}$, where $w\sk \in \HNN(A\sk, \theta\sk, B_1\sk, B_{-1}\sk)$ is the unitary implementing $\theta\sk$ in the HNN extension, for each $k \in \N \cup \{\infty\}$.
\end{col}

As a final application we describe here, one can also upgrade strong convergence to Shlyakhtenko's $A$-valued semicircular systems \cite{shlyakhtenko1999valued}, in a precise sense. This is in particular useful because it can be used to remove one of the hypotheses involving ambient strong convergence in Theorem B of the recent interesting work \cite{jekel2025strongconvergenceoperatorvaluedsemicirculars}.

\begin{col}\label{col: operator valued semicirculars}
    Fix index sets $I$ and $J$. Let $(X_i\sk)_{i \in I}$ be a tuple in a $C^\ast$-algebra $A\sk$ and $\eta\sk = (\eta_{j_1j_2}\sk)_{j_1, j_2 \in J}$ be a covariance matrix, for each $k \in \N \cup \{\infty\}$. Assume further that $A\si$ is generated by $(X_i\si)_{i \in I}$. Let $(S_j\sk)_{j \in J}$ be the $A\sk$-valued semicircular family with covariance $\eta\sk$, for each $k \in \N \cup \{\infty\}$. If $(X_i\sk; \eta\sk)_{i \in I}$ converges strongly to $(X_i\si; \eta\si)_{i \in I}$, then $(X_i\sk, S_j\sk)_{i \in I, j \in J}$ converges strongly to $(X_i\si, S_j\si)_{i \in I, j \in J}$.
\end{col}


\subsection*{Remarks to the reader} The proof of the main result, Theorem \ref{thm: main theorem}, is quite succinct using the gauge-invariant uniqueness theorem of Fowler-Muhly-Raeburn \cite{FowlerMuhlyRaeburn2003}, once the relevant setup and notations have been constructed precisely. One of the more notationally challenging aspects of these results is to set up the ambient weak and strong convergence hypotheses appropriately in various concrete settings (see Examples \ref{examps: examples for strong conv}). Once this is set up with care in amalgamated free products, Corollary \ref{col: amalgamated free products} is proved by fitting in \cite[Theorem 4.8.2]{brown2008textrm} which concretely links general reduced amalgamated free products to Toeplitz-Pimsner algebras.

As far as HNN extensions are concerned, we exploit the recent insight from \cite{gao2026selflessreducedamalgamatedfree} expressing the HNN extension as a concrete unital subalgebra of an amalgamated free product. Therefore, Corollary \ref{col: HNN extensions} can be proved by using Corollary \ref{col: amalgamated free products} appropriately. The proof of Corollary \ref{col: operator valued semicirculars} on the other hand is quite straightforward from the way the $A$-valued semicirculars are constructed, and is a direct application of Theorem \ref{thm: main theorem} in that setting.

\subsection{Acknowledgments} 
This work was done when the first author visited the second author, thanks to support from the Brin Mathematics Research Center at UMD College Park. We thank R. van Handel for an inspiring conversation that kickstarted this work. The paper is dedicated to Marius Junge and Leonel Robert with great admiration. Many of our recent exploits have greatly benefited from their vision.

\section{Preliminaries and Notations}

Unless otherwise mentioned, amalgamated free products of $C^\ast$-algebras equipped with nondegenerate conditional expectations will be in the reduced sense \cite{voiculescu1985symmetries}. Additionally, all tensor products will be minimal. All ultraproducts mentioned will be Banach space ultraproducts, which are naturally $C^\ast$-algebras (and are in fact the standard $C^\ast$-algebraic ultraproducts with respect to the operator norm) should the input Banach spaces all be $C^\ast$-algebras. Finally, if $\rho$ is a state on a $C^\ast$-algebra $A$, then $\Aut(A, \rho)$ shall denote the group of $\ast$-automorphisms of $A$ that preserve $\rho$.

We start with the definition of strong convergence for $C^\ast$-correspondences:

\begin{defn}
    Fix index sets $I$ and $J$. Let $(X_i\sk)_{i \in I}$ be a tuple in a $C^\ast$-algebra $A\sk$ and $(\xi_j\sk)_{j \in J}$ be a tuple in a $C^\ast$-correspondence $\cH\sk$ over $A\sk$, for each $k \in \N \cup \{\infty\}$. Assume further that $A\si$ is generated by $(X_i\si)_{i \in I}$ and $\cH\si$ is generated by $(\xi_j\si)_{j \in J}$ as a $C^\ast$-correspondence over $A\si$. We say that $(X_i\sk; \xi_j\sk)_{i \in I, j \in J}$ \emph{converges strongly} to $(X_i\si; \xi_j\si)_{i \in I, j \in J}$ if
    \begin{enumerate}
        \item $(X_i\sk)_{i \in I} \to (X_i\si)_{i \in I}$ strongly in the classical sense, i.e., for any $\ast$-polynomial $P$, $\lim_{k \to \infty} \|P(X_i\sk)\| = \|P(X_i\si)\|$; and,
        \item for any $\ast$-polynomials $P, Q$ and $j_1, j_2 \in J$,
        \begin{equation*}
            \lim_{k \to \infty} \|P(X_i\sk) - \langle \xi_{j_1}\sk, Q(X_i\sk)\xi_{j_2}\sk\rangle\| = \|P(X_i\si) - \langle \xi_{j_1}\si, Q(X_i\si)\xi_{j_2}\si\rangle\|.
        \end{equation*}
    \end{enumerate}
\end{defn}

\begin{rmk}\label{rmk: ultraprod ver of strong conv}
    Recall that, if we fix an index set $I$ and let $(X_i\sk)_{i \in I}$ be a tuple in a $C^\ast$-algebra $A\sk$, for each $k \in \N \cup \{\infty\}$. Assume further that $A\si$ is generated by $(X_i\si)_{i \in I}$. Then $(X_i\sk)_{i \in I}$ converges strongly to $(X_i\si)_{i \in I}$ in the classical sense if and only if, for each free ultrafilter $\cU$ on $\N$, there exists an embedding $\pi: A\si \to \prod_\cU A\sk$ s.t. $\pi(X_i\si) = (X_i\sk)_\cU$ for each $i \in I$. There is a corresponding ultraproduct characterization of strong convergence for $C^\ast$-correspondences, as follows: $(X_i\sk; \xi_j\sk)_{i \in I, j \in J} \to (X_i\si; \xi_j\si)_{i \in I, j \in J}$ strongly if and only if, for each free ultrafilter $\cU$ on $\N$, there exist embeddings $\pi: A\si \to \prod_\cU A\sk$ and $\phi: \cH\si \to \prod_\cU \cH\sk$ such that
    \begin{enumerate}
        \item $\pi(X_i\si) = (X_i\sk)_\cU$ for all $i \in I$ and $\phi(\xi_j\si) = (\xi_j\sk)_\cU$ for all $j \in J$;
        \item for $a, b \in A\si$, $\xi \in \cH\si$, if $\pi(a) = (a\sk)_\cU$, $\pi(b) = (b\sk)_\cU$, and $\phi(\xi) = (\xi\sk)_\cU$, then $\phi(a\xi b) = (a\sk\xi\sk b\sk)_\cU$; and,
        \item for $\xi, \eta \in \cH\si$, if $\phi(\xi) = (\xi\sk)_\cU$ and $\phi(\eta) = (\eta\sk)_\cU$, then $\pi(\langle\xi, \eta\rangle) = (\langle\xi\sk, \eta\sk\rangle)_\cU$.
    \end{enumerate}

    We leave it to the reader to check the above equivalence.
\end{rmk}

We will also use various standard notations in $C^\ast$-algebras. Please consult \cite{brown2008textrm}. In particular, recall the following:

\begin{defn}\label{defn: Toeplitz algebra defn}
    Let $\cH$ be a $C^\ast$-correspondence over a $C^\ast$-algebra $A$. Then the \emph{Toeplitz-Pimsner algebra} $\cT(\cH)$ associated to $\cH$ is defined to be the universal $C^\ast$-algebra generated by $A$ and $T_\xi$ (called \emph{creation operators}) for $\xi \in \cH$, subject to the relations that
    \begin{equation*}
        T_{\alpha\xi + \eta} = \alpha T_\xi + T_\eta\text{ for }\alpha \in \C, \xi, \eta \in \cH;
    \end{equation*}
    \begin{equation*}
        T_{a\xi b} = aT_\xi b\text{ for }a, b \in A, \xi \in \cH;
    \end{equation*}
    \begin{equation*}
        T_\xi^\ast T_\eta = \langle \xi, \eta\rangle\text{ for }\xi, \eta \in \cH.
    \end{equation*}

    Note that the above relation implies $\cH \ni \xi \mapsto T_\xi \in \cT(\cH)$ is an isometric $A$-$A$-bimodular linear map. From this it follows that, should $\cH$ be generated by $\{\xi_j\}_{j \in J}$, then $\cT(\cH)$ can alternatively be described as the universal $C^\ast$-algebra generated by $A$ and $T_{\xi_j}$ for $j \in J$, subject to the relations that $T_{\xi_{j_1}}^\ast aT_{\xi_{j_2}} = \langle \xi_{j_1}, a\xi_{j_2}\rangle$.

    From the definition, it is easy to check that there is an action $\beta: \T \curvearrowright \cT(\cH)$ defined by $\beta_z(a) = a$ for all $a \in A$ and $\beta_z(T_\xi) = zT_\xi$ for all $\xi \in \cH$. This is called the \emph{gauge action}.

    We also note that the Toeplitz-Pimsner algebra admits a canonical nondegenerate conditional expectation $E: \cT(\cH) \to A$ which satisfies $E(T_\xi T_\eta^\ast) = 0$ for all $\xi, \eta \in \cH$ \cite[Theorem 4.6.6(2)]{brown2008textrm}. In fact, by applying Cauchy-Schwarz, the reader may check that $E$ is uniquely characterized by the properties that it is positive, restricts to the identity map on $A$, and satisfies $E(T_\xi T_\xi^\ast) = 0$ for all $\xi \in \cH$.
\end{defn}

We recall the gauge-invariant uniqueness theorem for Toeplitz-Pimsner algebras \cite{FowlerMuhlyRaeburn2003} (see also \cite[Theorem 4.6.18]{brown2008textrm}).

\begin{thm}\label{thm: gauge-invariant uniqueness theorem}
    Let $\cH$ be a $C^\ast$-correspondence over a $C^\ast$-algebra $A$. Suppose $\cH$ is generated by $\{\xi_j\}_{j \in J}$. Let $B$ be a $C^\ast$-algebra, $\pi: A \to B$ be an embedding. Suppose $S_j \in B$ for $j \in J$ satisfies $S_{j_1}^\ast \pi(a)S_{j_2} = \pi(\langle \xi_{j_1}, a\xi_{j_2}\rangle)$. Suppose further that $B$ admits a gauge action, i.e., there is an action $\beta: \T \curvearrowright B$ s.t. $\beta_z(\pi(a)) = \pi(a)$ for all $a \in A$ and $\beta_z(S_j) = zS_j$ for all $j \in J$. Suppose also that
    \begin{equation}\label{eqn: intersection condition}
        \pi(A) \cap \overline{\Span}\{aS_{j_1}bS_{j_2}^\ast c: a, b, c \in A, j_1, j_2 \in J\} = \{0\}.
    \end{equation}

    Then $\pi$ extends to a (unique) embedding $\tilde{\pi}: \cT(\cH) \to B$ s.t. $\pi(T_{\xi_j}) = S_j$ for all $j \in J$.
\end{thm}

\begin{examps}\label{examps: examples for strong conv}
    \;
    \begin{enumerate}
        \item Consider $\cH = A$ as a $C^\ast$-correspondence over $A$, which is generated by $1$. By the gauge-invariant uniqueness theorem, $\cT(\cH) \cong A \otimes \cT$, where $\cT$ is the classical Toeplitz algebra (i.e., the universal $C^\ast$-algebra of an isometry), where $T_1 \in \cT(\cH)$ corresponds to the generating isometry of $\cT$.

        Suppose for each $k \in \N \cup \{\infty\}$ we have a $C^\ast$-algebra $A\sk$ and $\cH\sk = A\sk$. We leave it to the reader to check that $(X_i\sk; 1\sk)_{i \in I} \to (X_i\si, 1\si)_{i \in I}$ strongly if and only if $(X_i\sk)_{i \in I} \to (X_i\si)_{i \in I}$ strongly in the classical sense.
        \item For a more general example, consider a subalgebra $B \subset A$ admitting a nondegenerate conditional expectation $E: A \to B$. Then we can consider the $C^\ast$-correspondence $\cH = A \otimes_B A$, i.e., the completion of $A \odot A$ under the $A$-valued inner product
        \begin{equation*}
            \hspace{1.5 cm}\langle a \otimes b, c \otimes d\rangle = b^\ast E(a^\ast c)d.
        \end{equation*}

        $\cH$ is generated by the vector $1 \otimes 1$. By the gauge-invariant uniqueness theorem, $\cT(\cH)$ (we shall also denote it by $\cT(E)$ in this case) is isomorphic to $A \ast_B (B \otimes \cT)$, where $\cT$ is understood to be equipped with its vacuum state and where $T_{1 \otimes 1}$ corresponds to the generating isometry of $\cT$.

        Suppose for each $k \in \N \cup \{\infty\}$ we have an inclusion of $C^\ast$-algebras $B\sk \subset A\sk$ admitting a nondegenerate expectation $E\sk$ and $\cH\sk = A\sk \otimes_{B\sk} A\sk$. Again, we leave it to the reader to check that $(X_i\sk; (1 \otimes 1)\sk)_{i \in I} \to (X_i\si, (1 \otimes 1)\si)_{i \in I}$ strongly if and only if
        \begin{enumerate}
            \item $(X_i\sk)_{i \in I} \to (X_i\si)_{i \in I}$ strongly in the classical sense; and,
            \item for any $\ast$-polynomials $P, Q$,
            \begin{equation*}
                \hspace{1.8 cm}\lim_{k \to \infty} \|P(X_i\sk) - E\sk(Q(X_i\sk))\| = \|P(X_i\si) - E\si(Q(X_i\si))\|.
            \end{equation*}
        \end{enumerate}

        We shall take the above as the definition of $(X_i\sk; E\sk)_{i \in I}$ converging to $(X_i\si; E\si)_{i \in I}$ strongly (as elements of the $C^\ast$-probability spaces $(A\sk, E\sk)$ and $(A\si, E\si)$). We also leave it to the reader to check that the above definition is equivalent to the following: For each free ultarfilter $\cU$ on $\N$, there exists an embedding $\pi: A\si \to \prod_\cU A\sk$ such that
        \begin{enumerate}
            \item $\pi(X_i\si) = (X_i\sk)_\cU$ for all $i \in I$; and,
            \item for $a \in A\si$, if $\pi(a) = (a\sk)_\cU$, then $\pi(E\si(a)) = (E\sk(a\sk))_\cU$.
        \end{enumerate}
        \item For an even more general example, recall from \cite{shlyakhtenko1999valued} that, for an index set $J$, we say a matrix of maps $\eta = (\eta_{j_1j_2}: A \to A)_{j_1, j_2 \in J}$ is a \emph{covariance matrix} if for each finite $F \subset J$, $\eta|_F = (\eta_{j_1j_2})_{j_1, j_2 \in F}$ defines a completely positive map $A \to \M_{|F|}(A)$. We can then consider the $C^\ast$-correspondence $\cH(\eta)$ defined as the completion of $A \odot c_{00}(J; A)$ under the $A$-valued inner product
        \begin{equation*}
            \hspace{1.4 cm}\langle a \otimes b, c \otimes d\rangle = \langle b, \eta(a^\ast c)d\rangle_{c_{00}(J; A)} = \sum_{j_1, j_2 \in J} b_{j_1}^\ast\eta_{j_1j_2}(a^\ast c)d_{j_2}.
        \end{equation*}

        Note that $A \otimes_B A$ in the previous example is a special case of the construction here. Namely, it corresponds to the case where $J = \{\ast\}$ and $\eta = E$.

        For the general case, $\cH(\eta)$ is generated by $\{1 \otimes e_j\}_{j \in J}$. From this, it follows that $\cT(\cH(\eta))$ (we shall also denote it by $\cT(\eta)$ in this case) is the universal $C^\ast$-algebra generated by $A$ and $\{T_j\}_{j \in J}$ subject to the relations $T_{j_1}^\ast aT_{j_2} = \eta_{j_1j_2}(a)$. The $C^\ast$-algebra of \emph{the $A$-valued semicircular family with covariance $\eta$} is $C^\ast(A, S_j = T_j + T_j^\ast) \subset \cT(\eta)$, where $(S_j)_{j \in J}$ is the semicircular family.

        Fix index set $J$. Suppose for each $k \in \N \cup \{\infty\}$ we have a $C^\ast$-algebra $A\sk$ and a covariance matrix $(\eta_{j_1j_2}\sk)_{j_1, j_2 \in J}$. We once again leave it to the reader to check that $(X_i\sk; (1 \otimes e_j)\sk)_{i \in I, j \in J} \to (X_i\si; (1 \otimes e_j)\si)_{i \in I, j \in J}$ strongly if and only if
        \begin{enumerate}
            \item $(X_i\sk)_{i \in I} \to (X_i\si)_{i \in I}$ strongly in the classical sense; and,
            \item for any $\ast$-polynomials $P, Q$ and $j_1, j_2 \in J$,
            \begin{equation*}
                \hspace{1.8 cm}\lim_{k \to \infty} \|P(X_i\sk) - \eta_{j_1j_2}\sk(Q(X_i\sk))\| = \|P(X_i\si) - \eta_{j_1j_2}\si(Q(X_i\si))\|.
            \end{equation*}
        \end{enumerate}

        We shall take the above as the definition of $(X_i\sk; \eta\sk)_{i \in I} \to (X_i\si; \eta\si)_{i \in I}$ strongly. Moreover, the above definition is equivalent to the following: For each free ultarfilter $\cU$ on $\N$, there exists an embedding $\pi: A\si \to \prod_\cU A\sk$ such that
        \begin{enumerate}
            \item $\pi(X_i\si) = (X_i\sk)_\cU$ for all $i \in I$; and,
            \item for $a \in A\si$, $j_1, j_2 \in J$, if $\pi(a) = (a\sk)_\cU$, then $\pi(\eta_{j_1j_2}\si(a)) = (\eta_{j_1j_2}\sk(a\sk))_\cU$.
        \end{enumerate}
    \end{enumerate}
\end{examps}

\section{Proof of the Main Results}

\begin{proof}[Proof of Theorem \ref{thm: main theorem}]
    Fix a free ultrafilter $\cU$ on $\N$. By the strong convergence assumption, $\pi: A\si \to \prod_\cU A\sk$ defined by $\pi(X_j\si) = (X_j\sk)_\cU$ for all $j \in J$ is an embedding. Furthermore, by Remark \ref{rmk: ultraprod ver of strong conv}, we see that, for $a \in A\si$ and $j_1, j_2 \in J$, if $\pi(a) = (a\sk)_\cU$, then
    \begin{equation*}
        \left(T_{\xi_{j_1}\sk}^\ast a\sk T_{\xi_{j_2}\sk}\right)_\cU = \pi(\langle \xi_{j_1}\si, a\xi_{j_2}\si\rangle).
    \end{equation*}

    Thus, we have a well-defined $\ast$-homomorphism $\tilde{\pi}: \cT(\cH\si) \to \prod_\cU \cT(\cH\sk)$ extending $\pi$ and satisfying $\tilde{\pi}(T_{\xi_j\si}) = (T_{\xi_j\sk})_\cU$. The range admits a gauge action by applying the gauge action on each $\cT(\cH\sk)$. Equation (\ref{eqn: intersection condition}) in Theorem \ref{thm: gauge-invariant uniqueness theorem} can be verified by applying the canonical condition expectation $E\sk: \cT(\cH\sk) \to A\sk$ at each $k$. The result now follows from Theorem \ref{thm: gauge-invariant uniqueness theorem} and Remark \ref{rmk: ultraprod ver of strong conv}.
\end{proof}

\begin{rmk}\label{rmk: Toeplitz prob space strong conv}
    We also have the stronger result that $(X_i\sk, T_{\xi_j}\sk; E\sk)_{i \in I, j \in J}$ converges to $(X_i\si, T_{\xi_j}\si; E\si)_{i \in I, j \in J}$ strongly, where $E\sk: \cT(\cH\sk) \to A\sk$ is the canonical conditional expectation for each $k \in \N \cup \{\infty\}$. This follows from the characterization of the canonical conditional expectations given in the last paragraph of Definition \ref{defn: Toeplitz algebra defn}. We leave the details to the reader to check.
\end{rmk}

\begin{proof}[Proof of Corollary \ref{col: operator valued semicirculars}]
    This immediately follows from Theorem \ref{thm: main theorem} and part (3) of Examples \ref{examps: examples for strong conv}.
\end{proof}

For Corollary \ref{col: amalgamated free products}, we need to clarify the definitions first:

\begin{defn}
    Fix an index set $J$ and an index set $I_j$ for each $j \in J$. Let $D\sk \subset A_j\sk$ be an inclusion of $C^\ast$-algebras with nondegenerate condition expectation $E_j\sk$; and let $(X_{i, j}\sk)_{i \in I_j}$ be a tuple in $A_j\sk$, for each $j \in J$ and $k \in \N \cup \{\infty\}$. Assume further that $A_j\si$ is generated by $(X_{i, j}\si)_{i \in I_j}$ for every $j \in J$. We say $(X_{i, j}\sk; E_j\sk)_{i \in I_j} \to (X_{i, j}\si; E_j\si)_{i \in I_j}$ \emph{strongly for all $j \in J$, consistently on $D\sk$} if,
    \begin{enumerate}
        \item for each $j \in J$, $(X_{i, j}\sk; E_j\sk)_{i \in I_j} \to (X_{i, j}\si; E_j\si)_{i \in I_j}$ strongly in the sense defined in part (2) of Examples \ref{examps: examples for strong conv}; and,
        \item for any $\ast$-polynomials $P, Q$ and $j_1, j_2 \in J$,
        \begin{equation*}
            \lim_{k \to \infty} \|E_{j_1}\sk(P(X_{i_{j_1}, j_1}\sk)) - E_{j_2}\sk(Q(X_{i_{j_2}, j_2}\sk))\| = \|E_{j_1}\si(P(X_{i_{j_1}, j_1}\si)) - E_{j_2}\si(Q(X_{i_{j_2}, j_2}\si))\|.
        \end{equation*}
    \end{enumerate}
\end{defn}

\begin{rmk}\label{rmk: ultraprod char of consistent strong conv}
    We leave it to the reader to check that the definition above is equivalent to the following: For each free ultrafilter $\cU$ on $\N$, there exists an embedding $\pi_j: A_j\si \to \prod_\cU A_j\sk$ for each $j \in J$ such that
    \begin{enumerate}
        \item $\pi_j(X_{i, j}\si) = (X_{i, j}\sk)_\cU$ for all $i \in I_j, j \in J$;
        \item for $a \in A_j\si$, $j \in J$, if $\pi_j(a) = (a\sk)_\cU$, then $\pi_j(E_j\si(a)) = (E_j\sk(a\sk))_\cU$; and,
        \item for $d \in D\si$, $j_1, j_2 \in J$, we have $\pi_{j_1}(d) = \pi_{j_2}(d)$. (Note that condition 2 already implies $\pi_j(d) \in \prod_\cU D\sk$, so the equality here makes sense.)
    \end{enumerate}
\end{rmk}

The following lemma is the analogue of classical weak convergence in our context, and can be proved by a straightforward combinatorial argument similar to the proof of classical weak convergence. We leave the details to the reader to check.

\begin{lem}\label{lem: weak conv of AFP}
    If $(X_{i, j}\sk; E_j\sk)_{i \in I_j} \to (X_{i, j}\si; E_j\si)_{i \in I_j}$ strongly for all $j \in J$, consistently on $D\sk$, then for every free ultrafilter $\cU$ on $\N$, there is a well-defined $\ast$-homomorphism
    \begin{equation*}
        \phi: \ast^{\mathrm{alg}, j \in J}_{D\si} A_j\si \to \prod_\cU \ast^{j \in J}_{D\sk} A_j\sk
    \end{equation*}
    where $\ast^{\mathrm{alg}}$ denotes the algebraic amalgamated free product and $\phi(X_{i, j}\si) = (X_{i, j}\sk)_\cU$ for every $i \in I_j, j \in J$. Furthermore, $\phi$ preserves expectations, i.e., for every $t \in \ast^{\mathrm{alg}, j \in J}_{D\si} A_j\si$, if $\phi(t) = (t\sk)_\cU$, then $\phi(E_{D\si}(t)) = (E_{D\sk}(t\sk))_\cU$.
\end{lem}

We crucially need the following theorem from \cite{brown2008textrm}, which links general amalgamated free products to Toeplitz-Pimsner algebras:

\begin{thm}{\cite[Theorem 4.8.2]{brown2008textrm}}\label{thm: Toeplitz iso with AFP}
    Let $D \subset A_1$ and $D \subset A_2$ be inclusions of $C^\ast$-algebras admitting nondegenerate conditional expectations $E_1$ and $E_2$, resp. Consider the inclusion $D \oplus D \subset A_1 \oplus A_2$ with the expectation $E_1 \oplus E_2: A_1 \oplus A_2 \to D \oplus D$. Let $T = T_{1 \otimes 1} \in \cT(E_1 \oplus E_2)$, $p = 1 - T^2(T^\ast)^2$, $u = p(T + T^\ast)p$. Consider the map
    \begin{equation*}
        \Psi: A_1 \ast_D A_2 \to p\cT(E_1 \oplus E_2)p
    \end{equation*}
    defined by $\Psi(d) = \psi_1(d)$ where $d \in D$ and
    \begin{equation*}
        \Psi(a_1 \cdots a_n) = \psi_{j_1}(a_1) \cdots \psi_{j_n}(a_n)
    \end{equation*}
    where $a_l \in A_{j_l} \ominus D$ and $j_1 \neq j_2 \neq \cdots \neq j_n$; and where
    \begin{equation*}
        \psi_j: A_j \to p\cT(E_1 \oplus E_2)p
    \end{equation*}
    is defined by $\psi_i(a) = pap + uau$. Then $\Psi$ is a well-defined unital complete order isomorphism.
\end{thm}

\begin{proof}[Proof of Corollary \ref{col: amalgamated free products}]
    It clearly suffices to prove the case where $J = \{1, 2\}$. By the assumed strong convergence, we have $(X_{i_1, 1}\sk \oplus 0, 0 \oplus X_{i_2, 2}\sk; E_1\sk \oplus E_2\sk)_{i_1 \in I_1, i_2 \in I_2}$ converges to $(X_{i_1, 1}\si \oplus 0, 0 \oplus X_{i_2, 2}\si; E_1\si \oplus E_2\si)_{i_1 \in I_1, i_2 \in I_2}$ strongly.
    
    By Theorem \ref{thm: main theorem} (or more precisely Remark \ref{rmk: Toeplitz prob space strong conv}), we have
    \begin{equation*}
    \begin{split}
        &(X_{i_1, 1}\sk \oplus 0, 0 \oplus X_{i_2, 2}\sk, T\sk, p\sk, u\sk; (E_1\sk \oplus E_2\sk) \circ E_\cT\sk)_{i_1 \in I_1, i_2 \in I_2}\\
        \to &(X_{i_1, 1}\si \oplus 0, 0 \oplus X_{i_2, 2}\si, T\si, p\si, u\si; (E_1\si \oplus E_2\si) \circ E_\cT\si)_{i_1 \in I_1, i_2 \in I_2}
    \end{split}
    \end{equation*}
    strongly, where $T\sk = T_{1 \otimes 1}\sk \in \cT(E_1\sk \oplus E_2\sk)$, $p\sk = 1 - (T\sk)^2((T\sk)^\ast)^2$, $u\sk = p\sk(T\sk + (T\sk)^\ast)p\sk$, and $E_\cT\sk: \cT(E_1\sk \oplus E_2\sk) \to A_1\sk \oplus A_2\sk$ is the canonical condition expectation, for each $k \in \N \cup \{\infty\}$. We shall use $\Psi\sk: A_1\sk \ast_{D\sk} A_2\sk \to p\sk\cT(E_1\sk \oplus E_2\sk)p\sk$ to denote the map given in Theorem \ref{thm: Toeplitz iso with AFP} as applied to $D\sk \subset A_1\sk$ and $D\sk \subset A_2\sk$, for each $k \in \N \cup \{\infty\}$.

    For each $j = 1, 2$, fix a suitable universal $C^\ast$-algebra $C_j$ generated by variables labeled by elements of $I_j$. Here, by ``suitable", we meant that it is possible to evaluate elements of $C_j$ on $(X_{i, j}\sk)_{i \in I_j}$ for every $k \in \N \cup \{\infty\}$. Now, elements of the following form yields a dense subset of $A_1\si \ast_{D\si} A_2\si$:
    
    \begin{equation*}
        t = E_1\si[g(X_{i, 1}\si)] + \sum_r x_1\sr x_2\sr\cdots x_{n_r}\sr
    \end{equation*}
    where the sum is a finite sum, $g \in C_1$, and each product $x_1\sr x_2\sr\cdots x_{n_r}\sr$ is in reduced form, i.e., $x_l\sr = f_l\sr(X_{i, j_l\sr}\si) \in A_{j_l\sr}\si \ominus D\si$ where $l \in \{1, \cdots, n_r\}$, $f_l\sr \in C_{j_l\sr}$, $j_l\sr \in \{1, 2\}$, and $j_1\sr \neq \cdots \neq j_{n_r}\sr$.

    Using Lemma \ref{lem: weak conv of AFP}, the reader may check that it suffices to prove the following claim:

    \begin{claim}
        \;
        \begin{equation*}
            \lim_{k \to \infty} \left\|E_1\sk(g(X_{i, 1}\sk)) + \sum_r \prod_{\ell = 1}^{n_r} f_l\sr(X_{i, j_l\sr}\sk)\right\| = \|t\|
        \end{equation*}
        where the product goes from left to right.
    \end{claim}

    We now proceed to prove the above claim. For each $k \in \N \cup \{\infty\}$, we shall write
    \begin{equation*}
        \sharp(X_{i, j}\sk) = \begin{cases}
            X_{i, j}\sk \oplus 0,\text{ if }j = 1\\
            0 \oplus X_{i, j}\sk,\text{ if }j = 2
        \end{cases}.
    \end{equation*}
    
    By our assumptions, $\lim_{k \to \infty} \left\|E_{j_l\sr}\sk(f_l\sr(X_{i, j_l\sr}\sk))\right\| = 0$. Thus,
    \scriptsize
    \begin{equation}\label{eqn: main eqn for AFP}
    \begin{split}
        \Bigg\|\Psi\sk&\left(E_1\sk(g(X_{i, 1}\sk)) + \sum_r \prod_{\ell = 1}^{n_r} f_l\sr(X_{i, j_l\sr}\sk)\right) -\\
        &\Bigg[p\sk[E_1\sk(g(X_{i, 1}\sk)) \oplus 0]p\sk + u\sk[E_1\sk(g(X_{i, 1}\sk)) \oplus 0]u\sk +\\
        &\sum_r\prod_{l = 1}^{n_r} \Bigg(p\sk\left[f_l\sr\left(\sharp\left(X_{i, j_l\sr}\sk\right)\right)\right]p\sk + u\sk\left[f_l\sr\left(\sharp\left(X_{i, j_l\sr}\sk\right)\right)\right]u\sk\Bigg)\Bigg]\Bigg\| \to 0
    \end{split}
    \end{equation}
    \normalsize
    where again the products go from left to right. Since
    \begin{equation*}
    \begin{split}
        &(X_{i_1, 1}\sk \oplus 0, 0 \oplus X_{i_2, 2}\sk, T\sk, p\sk, u\sk; (E_1\sk \oplus E_2\sk) \circ E_\cT\sk)_{i_1 \in I_1, i_2 \in I_2}\\
        \to &(X_{i_1, 1}\si \oplus 0, 0 \oplus X_{i_2, 2}\si, T\si, p\si, u\si; (E_1\si \oplus E_2\si) \circ E_\cT\si)_{i_1 \in I_1, i_2 \in I_2}
    \end{split}
    \end{equation*}
    strongly, we must have the norm of the the term starting in the second line of Equation (\ref{eqn: main eqn for AFP}) converges to the norm of the same term with $k$ replaced by $\infty$. But that is exactly $\Psi\si(t)$, so
    \begin{equation*}
        \left\|\Psi\sk\left(E_1\sk(g(X_{i, 1}\sk)) + \sum_r \prod_{\ell = 1}^{n_r} f_l\sr(X_{i, j_l\sr}\sk)\right)\right\| \to \|\Psi\si(t)\|.
    \end{equation*}

    But $\Psi\sk$ and $\Psi\si$ are unital complete order isomorphisms, so in particular isometric. The claim follows.
\end{proof}

\begin{rmk}\label{rmk: application to groups}
    We document here an application of Corollary \ref{col: amalgamated free products} in the context of the recent work \cite{gao2026newsourcepurelyfinite}. In Lemma 2.3 therein, the setup consists of a group $G$, a subgroup $H < G$, a decreasing sequence of subgroups $H_n < G$ with $\bigcap_n H_n = H$, a normal subgroup $K_n \triangleleft G$ that is also contained in $H_n$ for each $n$, and another group $L$. The goal is to prove the sequence of group homomorphisms
    \begin{equation*}
        \pi_n: G \ast_H (H \times L) \to [G/K_n \ast_{H_n/K_n} (H_n/K_n \times L)] \times G
    \end{equation*}
    defined by sending $g \in G$ to $(gK_n, g)$ and $l \in L$ to $(l, e)$ induces an embedding
    \begin{equation}\label{eqn: embedding for groups}
        C_r^\ast(G \ast_H (H \times L)) \to \prod_\cU C_r^\ast([G/K_n \ast_{H_n/K_n} (H_n/K_n \times L)] \times G).
    \end{equation}

    To apply our Corollary \ref{col: amalgamated free products}, we note that both the LHS and the RHS of the above can be written as amalgamated free products. Indeed,
    \begin{equation*}
        C_r^\ast(G \ast_H (H \times L)) = C_r^\ast(G) \ast_{C_r^\ast(H)} C_r^\ast(H \times L)
    \end{equation*}
    and
    \begin{equation*}
    \begin{split}
        &C_r^\ast([G/K_n \ast_{H_n/K_n} (H_n/K_n \times L)] \times G)\\
        = &C_r^\ast(G/K_n \times G) \ast_{C_r^\ast(H_n/K_n \times G)} C_r^\ast(H_n/K_n \times G \times L).
    \end{split}
    \end{equation*}

    The reader may then check that
    \begin{equation*}
    \begin{split}
        &(\lambda_{(gK_n, g)}; E_{C_r^\ast(H_n/K_n \times G)})_{g \in G}\text{ in }C_r^\ast(G/K_n \times G)\\
        \to&(\lambda_g; E_{C_r^\ast(H)})_{g \in G}\text{ in }C_r^\ast(G)
    \end{split}
    \end{equation*}
    strongly; and
    \begin{equation*}
    \begin{split}
        &(\lambda_{(hK_n, h, e)}, \lambda_{(e, e, l)}; E_{C_r^\ast(H_n/K_n \times G)})_{h \in H, l \in L}\text{ in }C_r^\ast(H_n/K_n \times G \times L)\\
        \to&(\lambda_{(h, e)}, \lambda_{(e, l)}; E_{C_r^\ast(H)})_{h \in H, l \in L}\text{ in }C_r^\ast(H \times L)
    \end{split}
    \end{equation*}
    strongly. The two strong convergences are also consistent on $C_r^\ast(H_n/K_n \times G)$. The embedding in Equation (\ref{eqn: embedding for groups}) now follows from Corollary \ref{col: amalgamated free products}. We additionally emphasize to the reader that, while Corollary \ref{col: amalgamated free products} can replace, in the above manner, the isomorphism theorem \cite[Theorem 2.2]{gao2026newsourcepurelyfinite} used in the proof of \cite[Lemma 2.3]{gao2026newsourcepurelyfinite}, it by no means is a simplification of the argument. In the above, one still has to build on the key new insights of \cite{gao2026newsourcepurelyfinite} to use separable subgroups and Fell's absorption principle, allowing one to consider the appropriate group homomorphisms $\pi_n$ which are in a form suitable for the application of the machinery of Toeplitz-Pimsner algebras. Moreover, as we have seen above, the proof of the isomorphism theorem \cite[Theorem 2.2]{gao2026newsourcepurelyfinite} is much shorter and transparent than the proof of Corollary \ref{col: amalgamated free products} in the current work.
\end{rmk}

For HNN extensions, we again clarify the definitions first:

\begin{defn}
    Fix an index set $I$. Let $A\sk$ be a $C^\ast$-algebra equipped with a faithful state $\rho\sk$, $B_1\sk, B_{-1}\sk \subset A\sk$ be subalgebras admitting state-preserving conditional expectations $E_1\sk$ and $E_{-1}\sk$, resp., and $\theta\sk: B_1\sk \to B_{-1}\sk$ be a state-preserving $\ast$-isomorphism; and let $(X_i\sk)_{i \in I}$ be a tuple in $A\sk$, for each $k \in \N \cup \{\infty\}$. Assume further that $A\si$ is generated by $(X_i\si)_{i \in I}$. We say $(X_i\sk; \rho\sk, E_1\sk, E_{-1}\sk, \theta\sk)_{i \in I}$ \emph{converges strongly} to $(X_i\si; \rho\si, E_1\si, E_{-1}\si, \theta\si)_{i \in I}$ if,
    \begin{enumerate}
        \item for each $j \in \{\pm 1\}$, $(X_i\sk; E_j\sk)_{i \in I} \to (X_i\si; E_j\si)_{i \in I}$ strongly in the sense defined in part (2) of Examples \ref{examps: examples for strong conv};
        \item for any $\ast$-polynomial $P$,
        \begin{equation*}
            \lim_{k \to \infty} \rho\sk(P(X_i\sk)) = \rho\si(P(X_i\si));\text{ and,}
        \end{equation*}
        \item for any $\ast$-polynomials $P, Q$,
        \footnotesize
        \begin{equation*}
            \lim_{k \to \infty} \|\theta\sk(E_1\sk(P(X_i\sk))) - E_{-1}\sk(Q(X_i\sk))\| = \|\theta\si(E_1\si(P(X_i\si))) - E_{-1}\si(Q(X_i\si))\|.
        \end{equation*}
        \normalsize
    \end{enumerate}
\end{defn}

\begin{rmk}
    We leave it to the reader to check that the definition above is equivalent to the following: For each free ultrafilter $\cU$ on $\N$, there exists an embedding $\pi: A\si \to \prod_\cU A\sk$ such that
    \begin{enumerate}
        \item $\pi(X_i\si) = (X_i\sk)_\cU$ for all $i \in I$;
        \item for $a \in A\si$, if $\pi(a) = (a\sk)_\cU$, then $\lim_{k \to \cU} \rho\sk(a\sk) = \rho\si(a)$;
        \item for $a \in A\si$, $j \in \{\pm 1\}$, if $\pi(a) = (a\sk)_\cU$, then $\pi(E_j\si(a)) = (E_j\sk(a\sk))_\cU$; and,
        \item for $b \in B_1\si$, if $\pi(b) = (b\sk)_\cU$ with $b\sk \in B_1\sk$ for all $k$, then $\pi(\theta\si(b)) = (\theta\sk(b\sk))_\cU$.
    \end{enumerate}

    We note that condition (4) is equivalent to the version with $j = -1$ instead of $j = 1$.
\end{rmk}

To link HNN extensions with amalgamated free products, we need the following description of HNN extensions from \cite{gao2026selflessreducedamalgamatedfree}:

\begin{lem}{\cite[Remark 3.8]{gao2026selflessreducedamalgamatedfree}}\label{lem: structure of HNN}
    Consider $C = ((A \ast A) \rtimes \Z/2\Z) \ast_{(B_1 \ast B_{-1})} ((B_1 \ast B_{-1}) \rtimes \Z/2\Z)$ where the generator $u$ of the first $\Z/2\Z$ swaps the two free components and the generator $v$ of the second $\Z/2\Z$ swaps the two free components via $\theta$. Then $\HNN(A, \theta, B_1, B_{-1})$ is generated by the first free component $A$ in $C$ and $w = uv$, where $w$ is the unitary implementing $\theta$ in the HNN extension.
\end{lem}

\begin{proof}[Proof of Corollary \ref{col: HNN extensions}]
    We will follow the notations as in Lemma \ref{lem: structure of HNN}, with superscripts $(k)$, $k \in \N \cup \{\infty\}$, added. For clarity, let us write the two free copies of $A\sk$ in $A\sk \ast A\sk \subset C\sk$ as $A_1\sk$ and $A_{-1}\sk$, resp. For the tuple $(X_i\sk)_{i \in I}$ in $A\sk$, we will write $(X_{i, j}\sk)_{i \in I}$ for the tuple as elements in $A_j\sk$, $j = \pm 1$.

    By Corollary \ref{col: amalgamated free products}, $(X_{i, j}\sk)_{i \in I, j = \pm 1}$ in $A_1\sk \ast A_{-1}\sk$ converges strongly to $(X_{i, j}\si)_{i \in I, j = \pm 1}$ in $A_1\si \ast A_{-1}\si$. By a straightforward combinatorial argument similar to the classical proof of weak convergence, the reader may check that the assumptions on strong convergence of $E_1\sk$ and $E_{-1}\sk$ imply
    \begin{equation*}
        (X_{i, j}\sk; E_1\sk \ast E_{-1}\sk)_{i \in I, j = \pm 1} \to (X_{i, j}\si; E_1\si \ast E_{-1}\si)_{i \in I, j = \pm 1}
    \end{equation*}
    strongly.

    Note that for each $k \in \N \cup \{\infty\}$, $(A_1\sk \ast A_{-1}\sk) \rtimes \Z/2\Z$ is contained in $\M_2(A\sk \ast A\sk)$, where $A_1\sk \ast A_{-1}\sk$ is embedded as,
    \begin{equation*}
        \begin{bmatrix}
            A_1\sk \ast A_{-1}\sk & 0\\
            0 & A_{-1}\sk \ast A_1\sk
        \end{bmatrix}
    \end{equation*}
    and $u\sk$ is given by
    \begin{equation*}
        \begin{bmatrix}
            0 & 1\\
            1 & 0
        \end{bmatrix}
    \end{equation*}
    with the expectation onto $A_1\sk \ast A_{-1}\sk$ given by picking out the $(1, 1)$-corner. From this description and using the ultraproduct characterization of strong convergence (see the last paragraph of part (2) of Examples \ref{examps: examples for strong conv}), we see that
    \begin{equation}\label{eqn: HNN LHS strong conv}
        (X_{i, j}\sk, u\sk; E_{B_1\sk \ast B_{-1}\sk}\sk)_{i \in I, j = \pm 1} \to (X_{i, j}\si, u\si; E_{B_1\si \ast B_{-1}\si}\si)_{i \in I, j = \pm 1}
    \end{equation}
    strongly.

    On the other hand, the tuple $(E_1\si(P(X_{i, 1}\si)))_P$, where $P$ ranges over all $\ast$-polynomials, generates $B_1\si$; and the tuple $(\theta(E_1\si(P(X_{i, 1}\si))))_P$ generates $B_{-1}\si$. The assumptions on strong convergence of $\rho\sk$, $E_1\sk$, and $\theta\sk$ imply
    \begin{equation*}
        (E_1\sk(P(X_{i, 1}\sk)); \rho\sk)_P \to (E_1\si(P(X_{i, 1}\si)); \rho\si)_P
    \end{equation*}
    and
    \begin{equation*}
        (\theta\sk(E_1\sk(P(X_{i, 1}\sk))); \rho\sk)_P \to (\theta\si(E_1\si(P(X_{i, 1}\si))); \rho\si)_P
    \end{equation*}
    strongly.

    Now by Corollary \ref{col: amalgamated free products} and following similar arguments as in previous paragraphs for the $\Z/2\Z$-action, we obtain,
    \begin{equation}\label{eqn: HNN RHS strong conv}
    \begin{split}
        &(E_1\sk(P(X_{i, 1}\sk)), \theta\sk(E_1\sk(P(X_{i, 1}\sk))), v\sk; E_{B_1\sk \ast B_{-1}\sk}\sk)_P\\
        \to &(E_1\si(P(X_{i, 1}\si)), \theta\si(E_1\si(P(X_{i, 1}\si))), v\si; E_{B_1\si \ast B_{-1}\si}\si)_P
    \end{split}
    \end{equation}
    strongly.

    That the strong convergences in Equations (\ref{eqn: HNN LHS strong conv}) and (\ref{eqn: HNN RHS strong conv}) are consistent on the amalgams follows by unpacking the relevant definitions and using the ultraproduct characterization given in Remark \ref{rmk: ultraprod char of consistent strong conv}. The result now follows from Corollary \ref{col: amalgamated free products} and Lemma \ref{lem: structure of HNN}.
\end{proof}

\begin{rmk}\label{rmk: weak conv of HNN}
    We note that, using in addition a combinatorial argument similar to the classical proof of weak convergence, we can obtain the stronger conclusion that $(X_i\sk, w\sk; E_{A\sk}\sk)_{i \in I}$ converges strongly to $(X_i\si, w\si; E_{A\si}\si)_{i \in I}$. Here, $E_{A\sk}\sk$ is the conditional expectation from $\HNN(A\sk, \theta\sk, B_1\sk, B_{-1}\sk)$ onto $A\sk$, inherited from $C\sk$, for all $k \in \N \cup \{\infty\}$. We leave it to the reader to check the details.
\end{rmk}

\bibliographystyle{amsalpha}
\bibliography{bibliography}

@article{gao2026newsourcepurelyfinite,
      title={A new source of purely finite matricial fields}, 
      author={Gao, David and {Kunnawalkam Elayavalli}, Srivatsav and Manzoor, Aareyan and Patchell, Gregory},
      year={2026},
      journal={arXiv:2603.24502}
}

@article{male2011norm,
  title={The norm of polynomials in large random and deterministic matrices},
  author={Male, Camille},
  journal={Probability Theory and Related Fields},
  volume={152},
  number={1-2},
  pages={257--331},
  year={2011},
  publisher={Springer},
  doi={10.1007/s00440-011-0375-2}
}

@incollection{pimsner1997class,
  title={A class of {$C^*$}-algebras generalizing both {Cuntz-Krieger} algebras and crossed products by {$\mathbb{Z}$}},
  author={Pimsner, Mihai V},
  booktitle={Free probability theory ({Waterloo, ON, 1995})},
  volume={12},
  pages={189--212},
  year={1997},
  publisher={Amer. Math. Soc.},
  address={Providence, RI},
  series={Fields Inst. Commun.},
  doi={10.1090/fic/012/08}
}

@article{gao2026selflessreducedamalgamatedfree,
      title={Selfless reduced amalgamated free products and {HNN} extensions}, 
      author={David Gao and {Kunnawalkam Elayavalli}, Srivatsav and Gregory Patchell and Lizzy Teryoshin},
      year={2026},
      journal={arXiv: 2604.06982}
}

@article{skoufranis2015notion,
  title={On a notion of exactness for reduced free products of {$C^*$}-algebras},
  author={Skoufranis, Paul},
  journal={Journal f{\"u}r die reine und angewandte Mathematik (Crelles Journal)},
  volume={2015},
  number={700},
  pages={129--153},
  year={2015},
  publisher={De Gruyter},
  doi={10.1515/crelle-2013-0014}
}

@article {PisierRDP,
    AUTHOR = {Pisier, Gilles},
     TITLE = {Strong convergence for reduced free products},
   JOURNAL = {Infin. Dimens. Anal. Quantum Probab. Relat. Top.},
  FJOURNAL = {Infinite Dimensional Analysis, Quantum Probability and Related
              Topics},
    VOLUME = {19},
      YEAR = {2016},
    NUMBER = {2},
     PAGES = {1650008, 22},
      ISSN = {0219-0257,1793-6306},
   MRCLASS = {46L54 (46L07)},
  MRNUMBER = {3511889},
MRREVIEWER = {Franz\ Luef},
       DOI = {10.1142/S0219025716500089},
       URL = {https://doi.org/10.1142/S0219025716500089},
}

@article{FowlerMuhlyRaeburn2003,
  author  = {Neal J. Fowler and Paul S. Muhly and Iain Raeburn},
  title   = {Representations of {C}untz-{P}imsner algebras},
  journal = {Indiana Univ. Math. J.},
  volume  = {52},
  number  = {3},
  pages   = {569--605},
  year    = {2003},
  doi     = {10.1512/iumj.2003.52.2125},
  url     = {https://www.iumj.indiana.edu}
}

@article{shlyakhtenko1999valued,
  title={A-valued semicircular systems},
  author={Shlyakhtenko, Dimitri},
  journal={J. Funct. Anal.},
  volume={166},
  number={1},
  pages={1--47},
  year={1999},
  publisher={Elsevier},
  doi={10.1006/jfan.1999.3424},
  url={https://doi.org}
}

@article{HaagerupThorbjornsen2005,
  author  = {Haagerup, Uffe and Thorbj{\o}rnsen, Steen},
  title   = {A new application of random matrices: {E}xt({$C^*_{\text{red}}(\mathbb{F}_2)$}) is not a group},
  journal = {Ann. of Math.},
  series  = {Second Series},
  volume  = {162},
  number  = {2},
  pages   = {711--775},
  year    = {2005},
  publisher = {Ann. of Math.},
  doi     = {10.4007/annals.2005.162.711},
  url     = {https://annals.math.princeton.edu}
}

@article{vanhandel2025strongconvergenceshortsurvey,
      title={Strong convergence: a short survey}, 
      author={Ramon van Handel},
      year={2025},
      eprint={2510.12520},
      archivePrefix={arXiv},
      primaryClass={math.PR},
      url={https://arxiv.org/abs/2510.12520}, 
      journal={arXiv:2510.12520},
}

@article {KhinRX, AUTHOR = {Ricard, \'{E}ric and Xu, Quanhua}, TITLE = {Khintchine type inequalities for reduced free products and applications}, JOURNAL = {J. Reine Angew. Math.}, FJOURNAL = {Journal f\"{u}r die Reine und Angewandte Mathematik. [Crelle's Journal]}, VOLUME = {599}, YEAR = {2006}, PAGES = {27--59}, ISSN = {0075-4102}, MRCLASS = {46L09 (46L05 46L10)}, MRNUMBER = {2279097}, MRREVIEWER = {Maria Grazia Viola}, DOI = {10.1515/CRELLE.2006.077}, URL = {https://doi.org/10.1515/CRELLE.2006.077}, }

@article{Ueda2008Remarks,
  author    = {Ueda, Yoshimichi},
  title     = {Remarks on {HNN} extensions in operator algebras},
  journal   = {Illinois Journal of Mathematics},
  volume    = {52},
  number    = {3},
  pages     = {705--725},
  year      = {2008},
  doi       = {10.1215/ijm/1254143997}}

@article{magee2025strong,
  title   = {Strong convergence of unitary and permutation representations of discrete groups},
  author  = {Magee, Michael},
  journal = {arXiv:2503.21619},
  year    = {2025}}

@article{dykema2001exactness,
  title={Exactness of {Cuntz}--{Pimsner} {$C^*$}-algebras},
  author={Dykema, Ken and Shlyakhtenko, Dimitri},
  journal={Proc. Edinb. Math. Soc.},
  volume={44},
  number={2},
  pages={425--444},
  year={2001},
  publisher={Cambridge University Press}
}

@incollection{voiculescu1985symmetries,
  author    = {Voiculescu, Dan},
  title     = {Symmetries of some reduced free product {$C^*$}-algebras},
  booktitle = {Operator Algebras and Their Connections with Topology and Ergodic Theory},
  series    = {Lecture Notes in Mathematics},
  volume    = {1132},
  pages     = {556--588},
  year      = {1985},
  publisher = {Springer-Verlag},
  address   = {Berlin, Heidelberg},
  doi       = {10.1007/BFb0074909}
}

@article{ozawa2025proximalityselflessnessgroupcalgebras,
      title={Proximality and selflessness for group {$C^*$}-algebras}, 
      author={Narutaka Ozawa},
      year={2025},
      eprint={2508.07938},
      archivePrefix={arXiv},
      primaryClass={math.OA},
      journal={arXiv:2508.07938},
      url={https://arxiv.org/abs/2508.07938},
}

@article {robertselfless,
    AUTHOR = {Robert, Leonel},
     TITLE = {Selfless {$C^*$}-algebras},
   JOURNAL = {Adv. Math.},
  FJOURNAL = {Advances in Mathematics},
    VOLUME = {478},
      YEAR = {2025},
     PAGES = {Paper No. 110409},
      ISSN = {0001-8708,1090-2082},
   MRCLASS = {46L05 (46L35 46L54)},
       DOI = {10.1016/j.aim.2025.110409},
       URL = {https://doi.org/10.1016/j.aim.2025.110409},
}

@article{amrutam2025strictcomparisonreducedgroup,
      title={Strict comparison in reduced group {$C^*$}-algebras}, 
      author={Amrutam, Tattwamasi  and Gao, David and {Kunnawalkam Elayavalli}, Srivatsav and Patchell, Gregory},
      year={2025},
      journal={Invent. Math},
      volume={242},
      pages={639-657},
      doi={10.1007/s00222-025-01366-5},
      url={https://doi.org/10.1007/s00222-025-01366-5}
}

@article{jekel2025strongconvergenceoperatorvaluedsemicirculars,
      title={Strong convergence to operator-valued semicirculars}, 
      author={David Jekel and Yoonkyeong Lee and Brent Nelson and Jennifer Pi},
      year={2025},
      journal={arXiv:2506.19940}
}

@book{brown2008textrm,
  title={{$C^*$}-Algebras and Finite-Dimensional Approximations},
  author={Brown, Nathanial Patrick and Ozawa, Narutaka},
  volume={88},
  year={2008},
  publisher={Amer. Math. Soc.}
}

\end{document}